\documentclass[twoside,a4paper]{amsart}
\usepackage{amssymb,latexsym}
\input xypic
  \xyoption{all}

\usepackage{amsmath, amsthm, amsfonts}

\usepackage{pifont,pxfonts,txfonts}
\usepackage{yfonts}
\usepackage{bbm}
\usepackage{calrsfs}
\usepackage{empheq}
\usepackage{color}
\usepackage[colorlinks,linkcolor=blue,anchorcolor=blue,citecolor=blue]{hyperref}
\usepackage{amstext, amsbsy, amscd}
\usepackage[mathscr]{eucal}
\usepackage{times}
\usepackage{amsmath, amsthm, amsfonts,extpfeil}

\newtheorem{theorem}{Theorem}[section]
\newtheorem{prop}[theorem]{Proposition}
\newtheorem{lemma}[theorem]{Lemma}
\newtheorem{remark}[theorem]{Remark}

\newtheorem{corollary}[theorem]{Corollary}

\newcommand{\bR}{{\mathbb R}}

\newcommand{\sA}{{\mathscr A}}
\newcommand{\sF}{{\mathscr F}}

\newcommand{\sL}{{\mathscr L}}

\newcommand{\ba}{\begin{eqnarray}}
\newcommand{\na}{\end{eqnarray}}
\newcommand{\ban}{\begin{eqnarray*}}
\newcommand{\nan}{\end{eqnarray*}}

\newcommand{\suml}{\sum\limits}
\newcommand{\prodl}{\prod\limits}

\begin{document}

\title{On duality in symplectic cohomologies}

\author[Hua-Zhong Ke]{Hua-Zhong Ke}
\maketitle

\begin{center}
\emph{School of Mathematics, Sun Yat-sen University, Guangzhou,  510275, China}
\emph{kehuazh@mail.sysu.edu.cn}
\end{center}

{\bf Abstract:} For a symplectic manifold $(M^{2n},\omega)$ without boundary (not necessarily compact), we prove Poincar\'e type duality in filtered cohomology rings of differential forms on $M$, and we use this result to obtain duality between $(d+d^\Lambda)$- and $dd^\Lambda$-cohomologies.


{\bf MR(2010) Subject Classification:} 53D35

\date{\today}

\tableofcontents

\section{Introduction}

A symplectic manifold is a smooth manifold $M^{2n}$ equipped with a non-degenerate closed $2$-form $\omega$, i.e. $\omega^n$ is nowhere-vanishing and $d\omega=0$. By Darboux's theorem, $(M^{2n},\omega)$ is locally symplectomorphic to an open subset of $\bR^{2n}$ which is equipped with the standard symplectic form. So, in contrast to Riemannian geometry, symplectic invariants are necessarily global. The study of symplectic invariants has been very active in both mathematics and physics in the last few decades.

The well-known de Rham cohomology provides some of the most basic global topological invariants from differential forms. It is natural to ask whether one can obtain symplectic invariants from differential forms. On a closed Riemannian manifold, the celebrated Hodge decomposition theorem implies that each cohomology class admits a unique harmonic representative. This inspired Brylinski \cite{B} to conjecture that if $M$ is closed, then each cohomology class of $M$ admits a symplectic harmonic representative. Mathieu \cite{M} and Yan \cite{Y} showed that this conjecture holds on $M$ if and only if $M$ satisfies the strong Lefschetz property. Tseng and Yau initiated a program to search for new symplectic cohomologies of differential forms satisfying Hodge type decomposition. They first introduced $(d+d^\Lambda)$-, $dd^\Lambda$-, $\partial_+$- and $\partial_-$-cohomologies \cite{TY1,TY2}, and they proved that these cohomologies are naturally Lefschetz decomposable and hence are determined by their corresponding primitive cohomologies. Then Tsai, Tseng and Yau generalized these primitive cohomologies to filtered cohomologies \cite{TTY}. They found that filtered cohomologies give a two-sided resolution of Lefschetz maps. Moreover, they discovered a non-associative product operation on filtered forms, which generates an $A_3$-algebra structure on forms that underlies the filtered cohomologies and gives them a ring structure. Filtered cohomology rings contain both topological and symplectic information of $M$, and it is important to study these new symplectic invariants.

One of the most important properties of de Rham theory is the Poincar\'e duality, which states that the natural pairing $H^*(M)\times H^*_c(M)\rightarrow\bR$ gives $H^k(M)\cong(H_c^{2n-k}(M))^\vee$ for any $k$. Here the subscript "c" denotes cohomology of differential forms with compact supports. A natural question is whether filtered cohomology rings have such duality property. The main result of this note is the following.
\begin{prop}($=$Proposition \ref{mainprop})
The natural pairing $F^pH^*(M)\times F^pH_c^*(M)\rightarrow\bR$ gives 
\ban
F^pH^k(M)\cong(F^pH_c^{2n+2p-k}(M))^\vee,\forall k.
\nan
\end{prop}

When $M$ is closed, the above result was essentially proved in \cite{TY1,TY2} in the case $p=0$, by using harmonic representatives of these cohomologies. We take a different approach to deal with this problem. We first show that the duality property holds when $M$ is of finite type, by using the two-sided resolution of Lefschetz maps given by filtered cohomologies. Then we use the Mayer-Vietoris argument to prove the result in the general case, which is similar to the classical local-to-global proof of Poincar\'e duality in de Rham theory.

Note that primitive $(d+d^\Lambda)$- and $dd^\Lambda$-cohomologies are special cases of filtered cohomologies. So we can use the above proposition to prove the following duality between $(d+d^\Lambda)$- and $dd^\Lambda$-cohomologies, which generalizes Proposition 3.26 in \cite{TY1}. 
\begin{prop}($=$Proposition \ref{2ndmainprop}$+$Remark \ref{2ndmainrmk})
The natural pairing $H^*_{d+d^\Lambda}(M)\times H^*_{dd^\Lambda,c}(M)\rightarrow\bR$ gives
\ban
H^k_{d+d^\Lambda}(M)\cong(H^{2n-k}_{dd^\Lambda,c}(M))^\vee,\forall k,
\nan
and the natural pairing $H^*_{dd^\Lambda}(M)\times H^*_{d+d^\Lambda,c}(M)\rightarrow\bR$ gives
\ban
H^k_{dd^\Lambda}(M)\cong(H^{2n-k}_{d+d^\Lambda,c}(M))^\vee,\forall k.
\nan
\end{prop}

The rest of this note is arranged as follows. In Section 2, we recall some basic materials on filtered cohomology. In Section 3, we state the main result of this note, and discuss some immediate corollaries. In Section 4, we prove the main result in the finite type case, and in Section 5, we deal with the general case. In Section 6, we prove the duality between $(d+d^\Lambda)$- and $dd^\Lambda$-cohomologies.

\section{Preliminaries on filtered cohomology}

In this note, we always assume that $(M^{2n},\omega)$ is a symplectic manifold without boundary. It is well-known that the symplectic structure endows $\Omega^*$ with an $\mathfrak{sl}_2$ representation, where $\Omega^*$ is the space of differential forms on $M$. More precisely, let $L$ be the operation of wedging a form with $\omega$, let $\Lambda$ be the operation of contracting a form with the associated Poisson bivector field, and let $H$ be the operation of counting the degree of a homogeneous form:
\ban
H(A)=(n-k)A,\quad\forall A\in\Omega^k.
\nan
Then we have
\ban
[H,\Lambda]=2\Lambda,\quad[H,L]=-2L,\quad[\Lambda,L]=H.
\nan
In particular, $L$ is called the Lefschetz map.

We can decompose $\Omega^*$ into a direct sum of infinitely many irreducible finite-dimensional $\mathfrak{sl}_2$ representations. The highest weight states of these irreducible $\mathfrak{sl}_2$ representations are called primitive forms. In other words, a form $B_k\in\Omega^k$ with $k\leq n$ is called primitive if and only if it satisfies the following two equivalent conditions: (i) $\Lambda B_k=0$; (ii) $L^{n-k+1}B_k=0$. As a consequence of the above-mentioned $\mathfrak{sl}_2$ action on $\Omega^*$, we have the famous Lefschetz decomposition for forms on $M$, i.e. for each $A_k\in\Omega^k$ we can write uniquely
\ban
A_k=\suml_{\max\{0,k-n\}\leq l\leq\frac k2}L^lB_{k-2l},
\nan
where each $B_{k-2l}$ is primitive. We refer readers to Section 2.1 in \cite{TY2} for the explicit formula expressing $B_{k-2l}$ in terms of $A_k$. Let
\ban
\sL^{r,s}:=\{A_{2r+s}\in\Omega^{2r+s}|\textrm{ the Lefschetz decomposition for }A_{2r+s}\textrm{ is }A_{2r+s}=L^rB_s\}.
\nan
Then $\sL^{0,s}$ is precisely the space of primitive $s$-forms, and the Lefschetz decomposition gives
\ban
d=\partial_++L\partial_-,
\nan
where $\partial_\pm:\sL^{r,s}\rightarrow\sL^{r,s\pm1}$ are linear first order differential operators.


We recall three linear operators on $\Omega^*$ directly related to the Lefschetz decomposition \cite{TTY}, which will be used in this note. The first is the  reflection operator $*_r$ given by 
\ban
*_r(L^rB_s)=L^{n-r-s}B_s.
\nan
It is clear that $*_r$ is a projection operator, i.e. $*_r^2=\mathbbm 1$. Secondly, for each integer $p>0$, we define a linear operator $L^{-p}$ as follows. For any $A_k\in\Omega^k$, $L^{-p}A_k\in\Omega^{k-2p}$ comes from the Lefschetz decomposition for $A_k$:
\ban
A_k=\suml_{\max\{0,k-n\}\leq l\leq\frac k2}L^lB_{k-2l}\quad\Longrightarrow\quad L^{-p}A_k=\suml_{\max\{0,k-2p-n\}\leq l\leq\frac k2-p}L^lB_{k-2p-2l}.
\nan
One can check that for arbitrary $k$, we have $*_rA_k=L^{n-k}A_k$. Thirdly, for each integer $p\geq0$, we define a linear operator $\Pi^p$ as follows. For any $A_k\in\Omega^k$, $\Pi^{p}A_k\in\Omega^k$ comes from the Lefschetz decomposition for $A_k$:
\ban
A_k=\suml_{\max\{0,k-n\}\leq l\leq\frac k2}L^lB_{k-2l}\quad\Longrightarrow\quad \Pi^{p}A_k=\suml_{\max\{0,k-n\}\leq l\leq\min\{\frac k2, p\}}L^lB_{k-2l}.
\nan
One can check that we have the following relation of operators:
\ba\label{projLefschetz}
\mathbbm 1=\Pi^p+L^{p+1}L^{-p-1}.
\na

The space of $p$-filtered forms, denoted by $F^p\Omega^*$, is simply the image space of $\Pi^p$. The notion of filtered forms was first introduced in \cite{TTY} as a generalization of that of primitive forms. In other words, a form $A_k\in\Omega^k$ with $k\leq n+p$ is called $p$-filtered, if and only if it satisfies the following two equivalent conditions: (i)$\Lambda^{p+1}A_k=0$; (ii)$L^{n-k+1+p}A_k=0$. So a $0$-filtered form is nothing but a primitive form.

Inspired by previous results of Tseng and Yau \cite{TY1,TY2} on primitive cohomologies, for each integer $p$ with $0\leq p\leq n$, Tsai, Tseng and Yau \cite{TTY} introduced $p$-filtered cohomology $F^pH^*$, which is the cohomology of an elliptic complex $(\sF_p^k,d_k)$, where
\ban
\sF_p^k=\left\{\begin{array}{cc}
F^p\Omega^k,&0\leq k\leq n+p,\\
F^p\Omega^{\bar k},&n+p<k\leq 2n+2p+1,
\end{array}\right.
\nan
with $\bar k:=2n+2p+1-k$, and $d_k$'s are linear first or second order differential operators. Moreover, they discovered a novel graded commutative, non-associative product operation $\times$ on $\sF_p^*$, which generates an $A_3$-algebra structure on $\sF_p^*$. In particular, $(F^pH^*,\times)$ is a graded commutative algebra with identity. We refer readers to \cite{TTY} for details of the $A_3$-algebra structure.

Let $\Omega_c^*$ be the space of differential forms on $M$ with compact supports. Then $\Omega_c^*$ is an $\mathfrak{sl}_2$ sub-representation of $\Omega^*$. In particular, Lefschetz decomposition holds within $\Omega_c^*$. So we can consider compact filtered forms and compact filtered cohomology, and we will use a subscript "c" to denote the compact version. In particular, we point out that the product structure $\times$ on $\sF_p^*$ induces $F^pH^*_c$ with an $F^pH^*$-module structure.


\section{Duality in filtered cohomology}

We first define a linear function $\theta_p^M$ on $\sF_{c,p}^*$ by setting
\ban
\theta_p^M|_{\sF_{c,p}^k}=0,\quad0\leq k<2n+2p+1,
\nan
and
\ban
\theta_p^M(\sA_{2n+2p+1})=\int_M*_r\sA_{2n+2p+1},\quad\forall\sA_{2n+2p+1}\in\sF_{c,p}^{2n+2p+1}=\Omega_c^0.
\nan
One can check that 
\ba\label{stokes}
\theta_p^M(d_{2n+2p}\sA_{2n+2p})=0,\quad\forall\sA_{2n+2p}\in\sF_{c,p}^{2n+2p}.
\na
As a consequence, this linear function induces a linear function on $F^pH_c^{*}$, which we also denote by $\theta_p^M$. Note that $F^pH_c^*$ is an $F^pH^*$-module via the product structure $\times$ on $\sF_p^*$, and we define a bilinear pairing $g_p^M$ by the composition
\ban
F^pH^*\times F^pH_c^*\xrightarrow{\times}F^pH_c^*\xrightarrow{\theta_p^M}\bR.
\nan
This pairing $g_p^M$ gives a linear map
\ban
\Phi_p^M:F^pH^*\rightarrow(F^pH_c^*)^\vee,
\nan
which satisfies
\ban
\Phi_p^M([\sA_k])([\sA_{\bar k}])=g_p^M([\sA_k],[\sA_{\bar k}]),
\nan
for any $[\sA_k]\in F^pH^k$ and $[\sA_{\bar k}]\in F^pH_c^{\bar k}$.

The main result of this note is the following.
\begin{prop}\label{mainprop}
$\Phi_p^M$ is an isomorphism.
\end{prop}

We will prove Proposition \ref{mainprop} in Section \ref{finitetypecase} and \ref{generalcase}, and we will use it to establish duality between $(d+d^\Lambda)$- and $dd^\Lambda$-cohomologies in Section \ref{lastsection}. Before that, we discuss some corollaries. One immediate corollary is the following.
\begin{corollary}\label{nondegenerate}
The bilinear pairing $g_p^M$ is nondegenerate.
\end{corollary}
\begin{remark}
From linear algebra, if $F^pH^*$ and $F^pH_c^*$ are finite-dimensional, then Proposition \ref{mainprop} is equivalent to Corollary \ref{nondegenerate}. We will use this fact in Section \ref{finitetypecase}.
\end{remark}

If $M$ is closed, then $F^pH^*=F^pH_c^*$, and $g_p^M$ is a graded symmetric bilinear pairing on $F^pH^*$. Moreover, $g_p^M$ is compatible with the product $\times$ on $F^pH^*$:
\ban
g_p^M([\sA]\times[\sA'],[\sA''])=g_p^M([\sA],[\sA']\times[\sA'']).
\nan
So we have the following.
\begin{corollary}
If $M$ is closed, then $(F^pH^*,\times,\theta_p^M,g_p^M)$ is a graded Frobenius algebra.
\end{corollary}

\section{Finite type case}\label{finitetypecase}

In this section, we prove Proposition \ref{mainprop} under an additional assumption that $M$ is of finite type, that is, $M$ has a finite good cover \cite{BT}. 

Recall that $F^pH^*$ gives a two-sided resolution of the Lefschetz map $L^{p+1}$ (Theorem 4.2 in \cite{TTY}). In particular, we have the following short exact sequences
\ban
0\rightarrow Coker\bigg(H^{k-2p-2}\xrightarrow{L^{p+1}}H^k\bigg)\xrightarrow{\Pi^p}F^pH_+^k\xrightarrow{L^{-p-1}d}Ker\bigg(H^{k-2p-1}\xrightarrow{L^{p+1}}H^{k+1}\bigg)\rightarrow0,\\
0\rightarrow Coker\bigg(H^{2n-k-1}\xrightarrow{L^{p+1}}H^{2n-k+2p+1}\bigg)\xrightarrow{*_rdL^{-p-1}}F^pH_{-}^k\xrightarrow{*_r}Ker\bigg(H^{2n-k}\xrightarrow{L^{p+1}}H^{2n-k+2p+2}\bigg)\rightarrow0,
\nan
where 
\ban
F^pH_+^k:=F^pH^k,\quad F^pH_-^k:=F^pH^{\bar k},\quad 0\leq k\leq n+p.
\nan
Moreover, Theorem 4.2 in \cite{TTY} is algebraic in nature, and it also holds for differential forms with compact supports. So we have the following compact version:
\ban
0\rightarrow Coker\bigg(H_c^{k-2p-2}\xrightarrow{L^{p+1}}H_c^k\bigg)\xrightarrow{\Pi^p}F^pH_{c,+}^k\xrightarrow{L^{-p-1}d}Ker\bigg(H_c^{k-2p-1}\xrightarrow{L^{p+1}}H_c^{k+1}\bigg)\rightarrow0,\\
0\rightarrow Coker\bigg(H_c^{2n-k-1}\xrightarrow{L^{p+1}}H_c^{2n-k+2p+1}\bigg)\xrightarrow{*_rdL^{-p-1}}F^pH_{c,-}^k\xrightarrow{*_r}Ker\bigg(H_c^{2n-k}\xrightarrow{L^{p+1}}H_c^{2n-k+2p+2}\bigg)\rightarrow0.
\nan

Note that $M$ is of finite type. In such a case, $H^*$ and $H_c^*$ are finite-dimensional \cite{BT}, which implies that both $F^pH^*$ and $F^pH_c^*$ are also finite-dimensional. So to prove Proposition \ref{mainprop}, it suffices to show that $g_p^M$ is non-degenerate. To this end, from the definition of $g_p^M$, it is enough to show that for $0\leq k\leq n+p$,  $g_p^M|_{F^pH_+^k\times F^pH_{c,-}^k}$ and $g_p^M|_{F^pH_-^k\times F^pH_{c,+}^k}$ are both non-degenerate. We will only deal with the former case, since the proof of the latter case is similar.

We have the following non-canonical isomorphisms of vector spaces:
\ban
&&F^pH_+^k\cong Coker\bigg(H^{k-2p-2}\xrightarrow{L^{p+1}}H^k\bigg)\oplus Ker\bigg(H^{k-2p-1}\xrightarrow{L^{p+1}}H^{k+1}\bigg),\\
&&F^pH_{c,-}^k\cong Coker\bigg(H_c^{2n-k-1}\xrightarrow{L^{p+1}}H_c^{2n-k+2p+1}\bigg)\oplus Ker\bigg(H_c^{2n-k}\xrightarrow{L^{p+1}}H_c^{2n-k+2p+2}\bigg).
\nan
Observe that $H^*$ and $H_c^*$ are dual to each other, and the dual of $H^*\xrightarrow{L^{p+1}}H^*$ is exactly $H_c^*\xrightarrow{L^{p+1}}H_c^*$. So the following easy lemma shows that the dimensions of $F^pH_+^k$ and $F^pH_{c,-}^k$ are the same (see also Proposition 4.8 in \cite{TTY}).
\begin{lemma}
Let $V\xrightarrow{f}W$ be a linear map between vector spaces. Then we have the following canonical isomorphisms:
\ban
(Kerf)^\vee\cong Coker f^\vee,\quad(Cokerf)^\vee\cong Kerf^\vee.
\nan
\end{lemma}
\begin{proof}
Taking the dual of the exact sequence
\ban
0\rightarrow Kerf\rightarrow V\xrightarrow{f}W\rightarrow Cokerf\rightarrow0,
\nan
we get
\ban
0\rightarrow(Cokerf)^\vee\rightarrow W^\vee\xrightarrow{f^\vee}V^\vee\rightarrow(Ker f)^\vee\rightarrow0,
\nan
which gives the desired isomorphism.
\end{proof}

Since the dimensions of $F^pH_+^k$ and $F^pH_{c,-}^k$ are the same, it follows that we only need to prove the following: for any non-zero class $[A_k]\in F^pH_+^k$, we can find $[\bar A_k]\in F^pH_{c,-}^k$ such that
\ba
\int_MA_k\cdot*_r\bar A_k\neq0.\label{equivalent}
\na
Now we have the following two cases:
\begin{enumerate}
\item[(i)] $[L^{-p-1}dA_k]$ is the zero class in $H^{k-2p-1}$,
\item[(ii)]  $[L^{-p-1}dA_k]$ is a non-zero class in $H^{k-2p-1}$.
\end{enumerate}

\underline{Case (i):} In this case, there exists a non-zero class $[A_k']\in H^k$, such that  $A_k=\Pi^pA_k'$. Since $H^k\cong(H_c^{2n-k})^\vee$, it follows that we can find $[\bar A_k]\in F^pH_{c,-}^k$ such that 
\ban
\int_MA_k'\cdot *_r\bar A_k\neq 0.
\nan
Note that $\bar A_k$ is $p$-filtered, and we have the following Lefschetz decomposition
\ban
\bar A_k=\suml_{l\leq p}L^l\bar B_{k-2l},
\nan
which implies that
\ban
*_r\bar A_k=\suml_{l\leq p}L^{n-k+l}\bar B_{k-2l}.
\nan
Since each $\bar B_{k-2l}$ is primitive, it follows that
\ban
L^{p+1}*_r\bar A_k=\suml_{l\leq p}L^{n-(k-2l)+1+(p-l)}\bar B_{k-2l}=0.
\nan
From \eqref{projLefschetz}, we have $A_k'=A_k+L^{p+1}L^{-p-1}A_{k}'$, and 
\ban
\int_ML^{p+1}L^{-p-1}A_{k}'\cdot*_r\bar A_k=\int_ML^{-p-1}A_{k}'\cdot L^{p+1}*_r\bar A_k=0,
\nan
which implies \eqref{equivalent}.
 
\underline{Case (ii):} In this case, since $H^{k-2p-1}\cong(H_c^{2n-k+2p+1})^\vee$, it follows that there exists $[A_{2n-k+2p+1}']\in H_c^{2n-k+2p+1}$ such that
\ban
\int_ML^{-p-1}dA_k\cdot A_{2n-k+2p+1}'\neq0.
\nan
We want to show that 
\ban
\int_MdA_k\cdot L^{-p-1}A_{2n-k+2p+1}'\neq0.
\nan
To this end, we need the following.

\begin{prop}\label{productlemma}
Suppose that $2r+s+2r'+s'=2n$, and let $B_s\in\sL^{0,s}$, $B_{s'}\in\sL^{0,s'}$. If
\ban
L^rB_s\cdot L^{r'}B_{s'}\neq0,
\nan
then $s'=s$ and $r'=n-r-s$.
\end{prop}
\begin{proof}
We argue by contradiction. Without loss of generality, assume that $s'>s$. Since $2r+s+2r'+s'=2n$, it follows that $r'<n-r-s$.

Note that $B_{s'}$ is primitive, and then
\ban
L^{n-(s+s')+1+s}(B_sB_{s'})=B_s\cdot L^{n-s'+1}B_{s'}=0,
\nan
which implies that $B_sB_{s'}$ is $s$-filtered. Consider the Lefschetz decomposition
\ban
B_sB_{s'}=\suml_{l\leq s}L^lB_{s+s'-2l}'.
\nan
Then for each $l\leq s$, we have
\ban
&&r+r'+l\\
&=&\bigg(n-(s+s'-2l)+1\bigg)+(s+s'-2l-n-1)+(r+r'+l)\\
&=&\bigg(n-(s+s'-2l)+1\bigg)+(s+r+s'+r'-n)-l-1\\
&\geq&\bigg(n-(s+s'-2l)+1\bigg)+(n-r-r')-s-1\\
&\geq&n-(s+s'-2l)+1.
\nan
So
\ban
L^rB_s\cdot L^{r'}B_{s'}=\suml_{l\leq s}L^{r+r'+l}B_{s+s'-2l}'=0,
\nan
which contradicts the given condition.
\end{proof}
The above result immediately gives the following corollary.
\begin{corollary}\label{pairinglemma}
Let $B_s\in\sL^{0,s}$, and $B_{s'}\in\sL_c^{0,s'}$. If
\ban
\int_ML^rB_s\cdot L^{r'}B_{s'}\neq0,
\nan
then $s'=s$ and $r'=n-r-s$.
\end{corollary}
\begin{remark}
In Proposition \ref{productlemma} and Corollary \ref{pairinglemma}, we do not need the assumption that $M$ is of finite type.
\end{remark}

Now consider the Lefschetz decomposition
\ban
A_k&=&\suml_{l\leq p}L^lB_{k-2l},\\
A_{2n-k+2p+1}'&=&\suml_{l\geq n-k+2p+1}L^lB_{2n-k+2p+1-2l}'.\quad(\textrm{note that }k\leq n+p)
\nan
On the one hand, since
\ban
L^{-p-1}dA_k=\partial_-B_{k-2p}\in\sL^{0,k-2p-1},
\nan
it follows from Corollary \ref{pairinglemma} that
\ban
0\neq\int_ML^{-p-1}dA_k\cdot A_{2n-k+2p+1}'=\int_M\partial_-B_{k-2p}\cdot L^{n-k+2p+1}B_{k-2p-1}'.
\nan
On the other hand, observe that
\ban
dA_k=\suml_{l\leq p}L^l\partial_+B_{k-2l}+\suml_{l\leq p}L^{l+1}\partial_-B_{k-2l}=\suml_{l<p+1}L^l(\partial_+B_{k-2l}+\partial_-B_{k-2l+2})+L^{p+1}\partial_-B_{k-2p},
\nan
and
\ban
L^{-p-1}A_{2n-k+2p+1}'=\suml_{l\geq n-k+p}L^{l}B_{2n-k-1-2l}'=L^{n-k+p}B_{k-2p-1}'+\suml_{l>n-k+p}L^{l}B_{2n-k-1-2l}'.
\nan
Using Corollary \ref{pairinglemma} again, we see that
\ban
\int_MdA_k\cdot L^{-p-1}A_{2n-k+2p+1}'=\int_ML^{p+1}\partial_-B_{k-2p}\cdot L^{n-k+p}B_{k-2p-1}'=\int_M\partial_-B_{k-2p}\cdot L^{n-k+2p+1}B_{k-2p-1}'\neq0.
\nan
Setting $\bar A_k=*_rdL^{-p-1}A_{2n-k+2p+1}'$, we have \eqref{equivalent} by Stokes theorem.

\section{General case}\label{generalcase}

Based on the result in the last section, we will prove Proposition \ref{mainprop} in the general case. The approach is similar to that of the classical proof of Poincar\'e duality in de Rham cohomology.


Firstly, we discuss the functoriality of filtered cohomology with respect to symplectic open embedding, which is a symplectomorphism $\iota:U\rightarrow M$ from a symplectic manifold $U$ onto an open subset of $M$. One can check that the pullback operation $\iota^*$ on differential forms preserves $p$-filteredness:
\ban
\iota^*:\sF_p^*(M)\rightarrow\sF_p^*(U),
\nan
and it descends to $p$-filtered cohomology level: 
\ban
\iota^*:F^pH^*(M)\rightarrow F^pH^*(U).
\nan

On the other hand, we can pushforward compact differential forms on $U$ to $M$ by first identifying forms on $U$ with forms on $\iota(U)$ and then taking the zero extension to realize compact forms on $\iota(U)$ as compact forms on $M$. We denote the pushforward operation by $\iota_*$. One can check that $\iota_*$ also preserves $p$-filteredness:
\ban
\iota_*:\sF_{c,p}^*(U)\rightarrow\sF_{c,p}^*(M),
\nan
and it descends to $p$-filtered cohomology level: 
\ban
\iota_*:F^pH_c^*(U)\rightarrow F^pH_c^*(M).
\nan

The following "projection formulae" can be proved by tracing the definitions.
\begin{lemma}\label{projectionformulae}
Let $\iota:U\rightarrow M$ be a symplectic open embedding. For any $\sA^M\in\sF_p^*(M)$ and $\sA^U\in\sF_{c,p}^*(U)$, we have:
\begin{enumerate}
\item $\iota_*(\iota^*\sA^M\times\sA^U)=\sA^M\times\iota_*\sA^U$, and
\item $\theta_p^U(\iota^*\sA^M\times\sA^U)=\theta_p^M(\sA^M\times\iota_*\sA^U)$.
\end{enumerate}
\end{lemma}

One can use Lemma \ref{projectionformulae} to prove the following corollary, which is left to interested readers.
\begin{lemma}\label{functoriality}
Let $\iota:U\rightarrow M$ be a symplectic open embedding. Then we have the following commutative diagram:
\[
\begin{CD}
F^pH^k(M)@>{\iota^*}>>F^pH^k(U)\\
@V{\Phi_p^M}VV@VV{\Phi_p^U}V\\
(F^pH_c^{\bar k}(M))^\vee@>{\iota_*^\vee}>>(F^pH_c^{\bar k}(U))^\vee
\end{CD}
\] 
\end{lemma}

Secondly, we demonstrate the Mayer-Vietoris argument for filtered cohomologies. Let $U_0$ and $U_1$ be two nonempty open subsets of $M$ such that $M=U_0\cup U_1$. We set $U_{01}:=U_0\cap U_1$. Then pulling back $p$-filtered forms gives the following short exact sequence of chain complexes:
\ban
0\rightarrow\sF_p^*(M)\xrightarrow{r}\sF_p^*(U_0)\oplus\sF_p^*(U_1)\xrightarrow{\delta}\sF_0^*(U_{01})\rightarrow0,
\nan
where
\ban
r(\sA)&:=&(\sA|_{U_0},\sA|_{U_1}),\\
\delta(\sA^0,\sA^1)&:=&\sA^1|_{U_{01}}-\sA^0|_{U_{01}}.
\nan
The short exact sequence gives the following long exact sequence of $p$-filtered cohomologies:
\ban
\cdots\rightarrow F^pH^k(M)\xrightarrow{r}F^pH^k(U_0)\oplus F^pH^k(U_1)\xrightarrow{\delta}F^pH^k(U_{01})\xrightarrow{\partial}F^pH^{k+1}(M)\rightarrow\cdots.
\nan

Also pushing forward compact $p$-filtered forms gives the following short exact sequence of chain complexes:
\ban
0\rightarrow\sF_{c,p}^*(U_{01})\xrightarrow{\delta_c}\sF_{c,p}^*(U_0)\oplus\sF_{c,p}^*(U_1)\xrightarrow{s}\sF_{c,p}^*(M)\rightarrow0,
\nan
where
\ban
\delta_c(\sA)&:=&(-\sA,\sA),\\
s(\sA^0,\sA^1)&:=&\sA^0+\sA^1.
\nan
The short exact sequence gives the following long exact sequence of compact $p$-filtered cohomologies:
\ban
\cdots\rightarrow F^pH_c^k(U_{01})\xrightarrow{\delta_c}F^pH_c^k(U_0)\oplus F^pH_c^k(U_1)\xrightarrow{s}F^pH_c^k(M)\xrightarrow{\partial_c}F^pH_c^{k+1}(U_{01})\rightarrow\cdots.
\nan

Taking the dual of the second long exact sequence, and replacing $\partial$ by $\pm\partial$, we have the following row-exact diagram:
\[\minCDarrowwidth5pt
\begin{CD}
\cdots@>>>F^pH^k(M)@>r>>F^pH^k(U_0)\oplus F^pH^k(U_1)@>\delta>>F^pH^k(U_{01})@>(-1)^{k+1}\partial>>F^pH^{k+1}(M)@>>>\cdots\\
@.@V\Phi_p^MVV @V\Phi_p^{U_0}\oplus\Phi_p^{U_1}VV @V\Phi_p^{U_{01}}VV @V\Phi_p^MVV@. \\
\cdots@>>>(F^pH_c^{\bar k}(M))^\vee@>s^\vee>>(F^pH_c^{\bar k}(U_0))^\vee\oplus (F^pH_c^{\bar k}(U_1))^\vee@>\delta_c^\vee>>(F^pH_c^{\bar k}(U_{01}))^\vee@>\partial_c^\vee>>(F^pH_c^{\overline{k+1}}(M))^\vee@>>>\cdots
\end{CD}
\]
\begin{lemma}\label{MV}
The above diagram is commutative.
\end{lemma}
\begin{proof}
The commutativity of the leftmost two blocks comes from Lemma \ref{functoriality}. For the rightmost block, we need to trace the definitions of $\partial$ and $\partial_c$, and use \eqref{stokes} and the fact that the differential operator on $\sF_p^*$ satisfies the Leibnitz rule. Details are left to readers.

\end{proof}

Finally, we prove Proposition \ref{mainprop} in the general case. Let $\mathfrak U=\{U_\alpha\}_{\alpha\in I}$ be a good cover of $M$. From Section \ref{finitetypecase}, Proposition \ref{mainprop} holds on each $U_\alpha$ and any finite intersection of $U_\alpha$'s. So, from Lemma \ref{MV} and the Five Lemma, Proposition \ref{mainprop} holds on any finite union of $U_\alpha$'s. Moreover, we have the following result.

\begin{lemma}\label{countabledisjointunion}
Let $I'\subset I$ be an at most countable set, such that $U_\alpha\cap U_\beta=\emptyset$ for any distinct $\alpha,\beta\in I'$. Then Proposition \ref{mainprop} holds on $\bigcup\limits_{\alpha\in I'}U_\alpha$.
\end{lemma}
\begin{proof}
Note that $F^pH_c^*(\bigcup\limits_{\alpha\in I'}U_\alpha)=\bigoplus\limits_{\alpha\in I'}F^pH_c^*(U_\alpha)$, and the lemma comes directly from the following commutative diagram:
\[
\begin{CD}
F^pH^*(\bigcup\limits_{\alpha\in I'}U_\alpha)@>\Phi_p^{\bigcup\limits_{\alpha\in I'}U_\alpha}>>F^pH_c^*(\bigcup\limits_{\alpha\in I'}U_\alpha)^\vee\\
@|@|\\
\prodl_{\alpha\in I'}F^pH^*(U_\alpha)@>\prodl_{\alpha\in I'}\Phi_p^{U_\alpha}>>\prodl_{\alpha\in I'}F^pH_c^*(U_\alpha)^\vee
\end{CD}
\]
\end{proof}

We finish the proof of Proposition \ref{mainprop} by pointing out the following topological fact (see e.g. Chapter I in \cite{GHV}): there exist finitely many at most countable sets $I_1',\cdots,I_m'\subset I$, such that each $I_i'$ satisfies the condition of Lemma \ref{countabledisjointunion}, and $M=\bigcup\limits_{i=1}^m\bigcup\limits_{\alpha\in I_i'}U_\alpha$.

\section{Duality between $(d+d^\Lambda)$- and $dd^\Lambda$-cohomologies}\label{lastsection}

In this section, we state and prove duality between $(d+d^\Lambda)$- and $dd^\Lambda$-cohomologies. We remark that the duality was proved by using harmonic representatives when $M$ is closed (see Proposition 3.26 in \cite{TY1}).

We first recall some definitions. Let $d^\Lambda:=d\Lambda-\Lambda d$, and 
\ban
H^k_{d+d^\Lambda}:=\frac{Ker(d+d^\Lambda)\cap\Omega^k}{Imdd^\Lambda\cap\Omega^k},\quad H^k_{dd^\Lambda}:=\frac{Ker(dd^\Lambda)\cap\Omega^k}{Imd\cap\Omega^k+Imd^\Lambda\cap\Omega^k},\quad0\leq k\leq 2n.
\nan
An important property of $H^*_{d+d^\Lambda}$ and $H^*_{dd^\Lambda}$ is the Lefschetz decomposition. Namely, let
\ban
PH_{d+d^\Lambda}^k:=\frac{Ker d\cap\sL^{0,k}}{Im dd^\Lambda\cap\sL^{0,k}},\quad PH_{dd^\Lambda}^k:=\frac{Ker dd^\Lambda\cap\sL^{0,k}}{Im d\cap\sL^{0,k}+Im d^\Lambda\cap\sL^{0,k}},\quad 0\leq k\leq n,
\nan
and then Lefschetz decomposition on differential forms gives
\ban
H_{d+d^\Lambda}^k\cong\bigoplus\limits_{\max\{0,k-n\}\leq r\leq\frac k2}L^rPH_{d+d^\Lambda}^{k-2r},\quad H_{dd^\Lambda}^k\cong\bigoplus\limits_{\max\{0,k-n\}\leq r\leq\frac k2}L^rPH_{dd^\Lambda}^{k-2r}.
\nan
\begin{remark}
Theorem 3.11 and Theorem 3.22 in \cite{TY1} state that the Lefschetz decomposition for $H^*_{d+d^\Lambda}$ and $H^*_{dd^\Lambda}$ holds with the assumption that $M$ is closed, which is actually not needed as one can check by going through the proof.
\end{remark}

Consider the natural pairing on $M$:
\ban
\Omega^*\times\Omega_{c}^*&\rightarrow&\bR\\
(\sA,\sA')&\mapsto&\int_M\sA\cdot\sA'.
\nan
One can check that this induces the following pairing on symplectic cohomologies
\ban
H_{d+d^\Lambda}^*\times H_{dd^\Lambda,c}^*&\rightarrow&\bR\\
([\sA],[\sA'])&\mapsto&\int_M\sA\cdot\sA',
\nan
which induces the linear map
\ban
D:H_{d+d^\Lambda}^k\rightarrow(H_{dd^\Lambda,c}^{2n-k})^\vee,
\nan

Our aim in this section is to prove the following proposition.
\begin{prop}\label{2ndmainprop}
$D$ is a linear isomorphism.
\end{prop}
\begin{remark}\label{2ndmainrmk}
We can also interchange the roles of $d+d^\Lambda$ and $dd^\Lambda$ to obtain $H_{dd^\Lambda}^k\cong(H_{d+d^\Lambda,c}^{2n-k})^\vee$. We leave the proof to readers.
\end{remark}

To prove the above proposition, we first note that the Lefschetz decomposition also holds for $H_{dd^\Lambda,c}^k$:
\ban
H_{dd^\Lambda,c}^k\cong\bigoplus\limits_{\max\{0,k-n\}\leq r\leq\frac k2}L^rPH_{dd^\Lambda,c}^{k-2r}.
\nan
By Proposition \ref{productlemma} and Corollary \ref{pairinglemma}, we can decompose the linear map $D$ into the direct sum of the following maps:
\ba\label{decompD}
D_r:=D|_{L^rPH_{d+d^\Lambda}^{k-2r}}:L^rPH_{d+d^\Lambda}^{k-2r}\rightarrow(L^{n-k+r}PH_{dd^\Lambda,c}^{k-2r})^\vee.
\na
So we only need to show that \eqref{decompD} is a linear isomorphism. Recall that we have natural isomorphisms
\ban
L^{n-k+2r}PH_{d+d^\Lambda}^{k-2r}\cong F^{n-k+2r}H^{2n-k+2r+1},\quad L^{n-k+2r}PH_{dd^\Lambda,c}^{k-2r}\cong F^{n-k+2r}H_c^{2n-k+2r}.
\nan
Now consider the following diagram:
\[
\begin{CD}
L^rPH_{d+d^\Lambda}^{k-2r}@>D_r>>(L^{n-k+r}PH_{dd^\Lambda,c}^{k-2r})^\vee\\
@VL^{n-k+r}V\cong V@V\cong V(L^{-r})^\vee V\\
F^{n-k+2r}H^{2n-k+2r+1}@>\Phi_{n-k+2r}^M>\cong>(F^{n-k+2r}H_c^{2n-k+2r})^\vee
\end{CD}
\]
Here isomorphisms in the vertical directions and in the bottom can be checked directly. We show that this diagram is commutative, which implies that $D_r$ is an isomorphism. 

Let $[L^rB_{k-2r}]\in L^rPH_{d+d^\Lambda}^{k-2r}$ and $[L^{n-k+2r}B'_{k-2r}]\in L^{n-k+2r}PH_{dd^\Lambda,c}^{k-2r}$. On the one hand,
\ban
&&(L^{-r})^\vee(D_r[L^rB_{k-2r}])([L^{n-k+2r}B'_{k-2r}])\\
&=&(D_r[L^rB_{k-2r}])([L^{n-k+r}B'_{k-2r}])\\
&=&\int_ML^rB_{k-2r}\cdot L^{n-k+r}B'_{k-2r}.
\nan
On the other hand,
\ban
&&\Phi_{n-k+2r}^M(L^{n-k+r}[L^rB_{k-2r}])([L^{n-k+2r}B'_{k-2r}])\\
&=&\Phi_{n-k+2r}^M([L^{n-k+2r}B_{k-2r}])([L^{n-k+2r}B'_{k-2r}])\\
&=&\theta_{n-k+2r}^M([L^{n-k+2r}B_{k-2r}]\times[L^{n-k+2r}B'_{k-2r}])\\
&=&\int_M*_r(L^{n-k+2r}B_{k-2r})\cdot L^{n-k+2r}B'_{k-2r}\\
&=&\int_MB_{k-2r}\cdot L^{n-k+2r}B'_{k-2r},
\nan
where in the third equality, we have used the definition of $\times$ (see Definition 5.1 in \cite{TTY}). So the diagram is indeed commutative.

{\bf Acknowledgements.}
The author would like to thank Jianxun Hu for his constant support and encouragement, and Wei Yuan for helpful discussions. The author would also like to thank Li-Sheng Tseng for sharing his opinions through e-mail correspondence. The author is partially supported by NSFC grants 11521101 and 11601534.


\end{document}